\documentclass[12pt]{article}

\usepackage{amsfonts,amsmath,amsthm}
\usepackage{amssymb}
\usepackage{xcolor}
\usepackage{graphicx}
\usepackage{stmaryrd}
\usepackage{dsfont}

\usepackage[paper=a4paper,dvips,top=2cm,left=2cm,right=2cm,
foot=1cm,bottom=4cm]{geometry}

\newtheorem{Thm}{Theorem}[section]
\newtheorem{Def}[Thm]{Definition}

\newtheorem{remark}[Thm]{Remark}

\begin{document}

\title{The SOS Rank of a $5 \times 4$ Biquadratic Form via Orthogonality}

%\Large
\author{%
  % Liqun Qi\footnote{Jiangsu Provincial Scientific Research Center of Applied Mathematics, Nanjing 211189, China. Department of Applied Mathematics, The Hong Kong Polytechnic University, Hung Hom, Kowloon, Hong Kong. 			({\tt maqilq@polyu.edu.hk})}
  %\and
        Yannan Chen\footnote{School of Mathematical Sciences, South China Normal University, Guangzhou 510631, China. (\texttt{email:~ynchen@scnu.edu.cn})}
  %\and
  %	Chunfeng Cui\footnote{School of Mathematical Sciences, Beihang University, Beijing  100191, China. {\tt chunfengcui@buaa.edu.cn})}
  %\and {and \
  % Yi Xu\footnote{School of Mathematics, Southeast University, Nanjing  211189, China. Nanjing Center for Applied Mathematics, Nanjing 211135,  China. Jiangsu Provincial Scientific Research Center of Applied Mathematics, Nanjing 211189, China. ({\tt yi.xu1983@hotmail.com})}		}
}

\date{\today}
\maketitle

\begin{abstract}

Biquadratic forms arise naturally in polynomial optimization, tensor analysis, and quantum information theory. A key problem is determining the minimal number of squares needed in a sum-of-squares (SOS) representation of such a form, known as its SOS rank. For fixed dimensions $(m,n)$, the maximum possible SOS rank over all biquadratic forms in
$m$ and $n$ variables is denoted $\operatorname{BSR}(m,n)$. Recent advances have established lower bounds on $\operatorname{BSR}(m,n)$ via combinatorial constructions involving bipartite graphs and the orthogonality method. In particular, for the case $(m,n)=(5,4)$, it was shown that $\operatorname{BSR}(5,4)\ge 11$ using only nondegenerate 2-edges. In this paper, we extend this framework by incorporating a degenerate $2$-edge, which introduces a cross term where the two $y$-indices coincide. We construct an explicit $5\times 4$ biquadratic form and apply the orthogonality method to prove that its SOS rank is $12$, thereby improving the lower bound to $\operatorname{BSR}(5,4)\ge12$. This result demonstrates that degenerate $2$-edges yield additional algebraic flexibility beyond purely combinatorial bounds and extends the applicability of the orthogonality method to forms with cross terms involving identical $y$-indices.

\medskip


\end{abstract}

\newpage

\section{Introduction}

Biquadratic forms, polynomials of the form
\[
P(\mathbf{x},\mathbf{y}) = \sum_{i=1}^m \sum_{j=1}^n \sum_{k=1}^m \sum_{l=1}^n a_{ijkl} \, x_i x_k y_j y_l
\]
with real coefficients, arise naturally in numerous areas of mathematics and applications, including polynomial optimization, tensor analysis, and quantum information theory. A fundamental question concerning such forms is whether they admit a representation as a sum-of-squares (SOS) of bilinear forms, and if so, what is the minimal number of squares required. This minimal number, denoted $\operatorname{sos}(P)$, is known as the \emph{SOS rank} of $P$. Determining the SOS rank of a given biquadratic form, or understanding its possible range over all forms of fixed dimensions, has attracted considerable recent attention \cite{BPSV19,BSSV22,CQX26,QCX26,XCQ26}.

For a fixed pair of positive integers $(m,n)$, define
\[
\operatorname{BSR}(m,n) = \max\{ \operatorname{sos}(P) : P \text{ is a biquadratic form in } m \text{ and } n \text{ variables} \}.
\]
This quantity measures the maximum complexity of SOS representations for biquadratic forms of given dimensions. While some upper and lower bounds for  $\operatorname{BSR}(m,n)$ were recently established, the exact value of $\operatorname{BSR}(m,n)$ remains unknown for most pairs $(m,n)$. In particular, determining the extremal forms that achieve the maximum SOS rank is a challenging problem with deep connections to extremal combinatorics, specifically to the classical Zarankiewicz problem on bipartite graphs with forbidden subgraphs.

A powerful technique for obtaining lower bounds on $\operatorname{sos}(P)$ (and hence on $\operatorname{BSR}(m,n)$) was developed in Theorem~4.1 of \cite{XCQ26}. This \emph{orthogonality method} associates to each monomial $x_i^2 y_j^2$ a vector $\mathbf{v}_{ij}$ in $\mathbb{R}^r$, where $r$ is the number of squares in a hypothetical SOS decomposition. The coefficients of $P$ impose norm constraints on these vectors via pure square terms, while missing cross terms force orthogonality relations among them. By constructing a sufficiently large set of pairwise orthogonal nonzero vectors, one deduces a lower bound $r \ge k$, thereby establishing $\operatorname{sos}(P) \ge k$.

Recent work has established a combinatorial framework for constructing biquadratic forms with large SOS rank using bipartite graphs \cite{CQX26,QCX26}. A significant advance in this direction is the introduction of the \emph{double Zarankiewicz number with nondegenerate $2$-edges}, denoted $z_{2}^{\text{non}}(m,n)$, in \cite{QCX26}. This concept extends the classical Zarankiewicz number by allowing a graph to contain both ordinary $1$-edges, which correspond to pure squares $x_i^2 y_j^2$, and nondegenerate $2$-edges, which correspond to squares of bilinear forms $(x_i y_j + x_k y_l)^2$ with $i \neq k$ and $j \neq l$. For such a graph $G$ that avoids certain generalized $C_4$ cycles, the associated doubly simple biquadratic form satisfies $\operatorname{sos}(P_G) = |E_1| + |E_2|$. Consequently, one obtains the general lower bound $\operatorname{BSR}(m,n) \ge z_{2}^{\text{non}}(m,n)$.

For the case $m = 5$, $n = 4$, it was determined in \cite{QCX26} that the maximum total number of edges achievable with nondegenerate $2$-edges is $z_{2}^{\text{non}}(5,4) = 11$, yielding the bound $\operatorname{BSR}(5,4) \ge 11$. However, this construction uses only nondegenerate $2$-edges, where the two $y$-indices are distinct. A natural question arises: can a higher SOS rank be attained by allowing \emph{degenerate} $2$-edges, i.e., squares of the form $(x_i y_j + x_k y_j)^2 = (x_i + x_k)^2 y_j^2$, where the two $y$-indices coincide? Such terms introduce cross terms $2x_i x_k y_j^2$ that were previously excluded from the nondegenerate framework.

In this paper, we answer this question affirmatively by constructing an explicit $5 \times 4$ biquadratic form that incorporates a degenerate $2$-edge. We apply the orthogonality method to this form and prove that its SOS rank is $12$, thereby establishing the improved lower bound
\[
\operatorname{BSR}(5,4) \ge 12.
\]
This result is significant for two reasons. First, it demonstrates that degenerate $2$-edges provide additional algebraic flexibility that can increase the SOS rank beyond the purely combinatorial limit imposed by nondegenerate constructions. Second, it shows that the orthogonality method from \cite{XCQ26} is robust enough to handle cross terms where $j = l$, a situation not covered in the original nondegenerate analysis. The form we analyze is notable because it is neither \emph{simple} (it contains cross terms) nor \emph{$y$-deficient} (it includes a cross term $2x_4 x_5 y_2^2$ with identical $y$-indices), yet the method applies seamlessly.

The paper is organized as follows. Section~2 presents the $5 \times 4$ biquadratic form, provides an explicit SOS decomposition into $12$ squares, and uses the orthogonality method to prove that no decomposition into fewer squares exists. Concluding remarks are given in Section~3.

\section{A Twelve-Square Form in $5 \times 4$ Variables}

In this section we apply the orthogonality method developed in Theorem 4.1 of \cite{XCQ26} to a $5 \times 4$ biquadratic form. The analysis follows the same systematic approach: we first exhibit an explicit SOS decomposition to obtain an upper bound, then use the vector orthogonality technique to prove a matching lower bound. The result demonstrates the versatility of the method across different dimensions.

\begin{Def}\label{def:5x4-form}
Define the $5 \times 4$ biquadratic form
\[
\begin{aligned}
P(\mathbf{x},\mathbf{y}) =&\; x_1^2 y_1^2 + x_1^2 y_2^2 + x_1^2 y_3^2 \\
&+ x_2^2 y_1^2 + x_2^2 y_4^2 \\
&+ x_3^2 y_2^2 + x_3^2 y_4^2 \\
&+ x_4^2 y_3^2 + x_4^2 y_4^2 \\
&+ x_5^2 y_1^2 \\
&+ (x_2 y_3 + x_5 y_4)^2 + (x_4 + x_5)^2 y_2^2,
\end{aligned}
\]
where $\mathbf{x} = (x_1,\dots,x_5)$ and $\mathbf{y} = (y_1,\dots,y_4)$.
\end{Def}

Expanding the two squared terms gives
\[
\begin{aligned}
(x_2 y_3 + x_5 y_4)^2 &= x_2^2 y_3^2 + x_5^2 y_4^2 + 2 x_2 x_5 y_3 y_4,\\
(x_4 + x_5)^2 y_2^2 &= x_4^2 y_2^2 + x_5^2 y_2^2 + 2 x_4 x_5 y_2^2.
\end{aligned}
\]
Thus the complete expansion of $P$ contains the following pure square terms $x_i^2 y_j^2$:
\[
\begin{array}{c|c}
(i,j) & \text{coefficient} \\ \hline
(1,1),\;(1,2),\;(1,3) & 1 \\
(2,1),\;(2,3),\;(2,4) & 1 \\
(3,2),\;(3,4) & 1 \\
(4,2),\;(4,3),\;(4,4) & 1 \\
(5,1),\;(5,2),\;(5,4) & 1
\end{array}
\]
All other $(i,j)$ with $1 \le j \le 4$ have coefficient $0$. The only nonzero cross terms $x_i x_k y_j y_l$ with $i \neq k$ are
\[
2 x_2 x_5 y_3 y_4 \quad \text{and} \quad 2 x_4 x_5 y_2^2,
\]
each with coefficient $2$. Note that the second cross term has $j = l = 2$, i.e., the same $y$-index in both factors.

\begin{Thm}\label{thm:5x4-rank12}
The biquadratic form $P$ defined in Definition~\ref{def:5x4-form} satisfies $\operatorname{sos}(P) = 12$.
\end{Thm}

\begin{proof}
\textbf{Upper bound.}
The definition itself provides an explicit SOS decomposition into $12$ squares:
\[
\begin{aligned}
P =&\; (x_1 y_1)^2 + (x_1 y_2)^2 + (x_1 y_3)^2 \\
&+ (x_2 y_1)^2 + (x_2 y_4)^2 \\
&+ (x_3 y_2)^2 + (x_3 y_4)^2 \\
&+ (x_4 y_3)^2 + (x_4 y_4)^2 \\
&+ (x_5 y_1)^2 \\
&+ (x_2 y_3 + x_5 y_4)^2 + (x_4 y_2 + x_5 y_2)^2.
\end{aligned}
\]
Thus we say $\operatorname{sos}(P) \le 12$.

\textbf{Lower bound.}
Assume that $P$ admits a decomposition into $r$ squares of bilinear forms with $r \le 12$:
\[
P(\mathbf{x},\mathbf{y}) = \sum_{k=1}^{r} \ell_k(\mathbf{x},\mathbf{y})^2, \qquad
\ell_k(\mathbf{x},\mathbf{y}) = \sum_{i=1}^{5} \sum_{j=1}^{4} a_{ij}^{(k)} x_i y_j.
\]
For each pair $(i,j)$ define the vector $\mathbf{v}_{ij} = \left(a_{ij}^{(1)}, \dots, a_{ij}^{(r)}\right)^T \in \mathbb{R}^r$. From the expansion of $\left(\sum a_{ij}^{(k)} x_i y_j\right)^2$, the coefficient of $x_i^2 y_j^2$ equals $\|\mathbf{v}_{ij}\|^2$, and for $i \neq k$, $j \neq l$ the coefficient of $x_i x_k y_j y_l$ equals $2(\mathbf{v}_{ij} \cdot \mathbf{v}_{kl} + \mathbf{v}_{il} \cdot \mathbf{v}_{kj})$; for $j = l$ the coefficient equals $2\,\mathbf{v}_{ij} \cdot \mathbf{v}_{kj}$, where the small dot ``$\cdot$'' stands for the dot product of vectors.

From the pure square coefficients we obtain:
\begin{itemize}
    \item For all present $(i,j)$ listed above: $\|\mathbf{v}_{ij}\| = 1$.
    \item For absent $(i,j)$: $\mathbf{v}_{ij} = \mathbf{0}$.
\end{itemize}
The zero vectors (absent pure squares) are:
\[
\mathbf{v}_{14}=\mathbf{v}_{22}=\mathbf{v}_{31}=\mathbf{v}_{33}=\mathbf{v}_{41}=\mathbf{v}_{53} = \mathbf{0}.
\]

Now apply the cross-term relations.

\textbf{Cross term $2x_2 x_5 y_3 y_4$ ($j=3$, $l=4$, $j \ne l$):}
\[
2(\mathbf{v}_{23} \cdot \mathbf{v}_{54} + \mathbf{v}_{24} \cdot \mathbf{v}_{53}) = 2.
\]
Since $\mathbf{v}_{53} = \mathbf{0}$, we have $\mathbf{v}_{23} \cdot \mathbf{v}_{54} = 1$. With $\|\mathbf{v}_{23}\| = \|\mathbf{v}_{54}\| = 1$, Cauchy-Schwarz inequality yields
\[
\mathbf{v}_{23} = \mathbf{v}_{54}.
\]

\textbf{Cross term $2x_4 x_5 y_2^2$ ($j=l=2$):}
\[
2\,\mathbf{v}_{42} \cdot \mathbf{v}_{52} = 2 \quad\Longrightarrow\quad \mathbf{v}_{42} \cdot \mathbf{v}_{52} = 1.
\]
With $\|\mathbf{v}_{42}\| = \|\mathbf{v}_{52}\| = 1$, we obtain
\[
\mathbf{v}_{42} = \mathbf{v}_{52}.
\]

We now derive orthogonality relations from missing cross terms.

\noindent\textit{Within the same row.}
For $i=1$: missing $x_1^2 y_1 y_2$, $x_1^2 y_1 y_3$, $x_1^2 y_2 y_3$ give
\[
\mathbf{v}_{11} \cdot \mathbf{v}_{12}=0,\quad \mathbf{v}_{11} \cdot \mathbf{v}_{13}=0,\quad \mathbf{v}_{12} \cdot \mathbf{v}_{13}=0.
\]
For the convenience of readers, we rewrite these vertical relationships in the following form.\\
\textbf{Row 1 orthogonality}
\[
\begin{array}{c|c|c}
\text{Pair} & \text{Missing cross term} & \text{Justification} \\ \hline
\mathbf{v}_{11} \perp \mathbf{v}_{12} & x_1^2 y_1 y_2 & \text{row 1} \\
\mathbf{v}_{11} \perp \mathbf{v}_{13} & x_1^2 y_1 y_3 & \text{row 1} \\
\mathbf{v}_{12} \perp \mathbf{v}_{13} & x_1^2 y_2 y_3 & \text{row 1}
\end{array}
\]
For $i=2,3,4,5$, we could establish similar vertical relationships.\\
\textbf{Row 2 orthogonality}
\[
\begin{array}{c|c|c}
\text{Pair} & \text{Missing cross term} & \text{Justification} \\ \hline
\mathbf{v}_{21} \perp \mathbf{v}_{23} & x_2^2 y_1 y_3 & \text{row 2} \\
\mathbf{v}_{21} \perp \mathbf{v}_{24} & x_2^2 y_1 y_4 & \text{row 2} \\
\mathbf{v}_{23} \perp \mathbf{v}_{24} & x_2^2 y_3 y_4 & \text{row 2}
\end{array}
\]
\textbf{Rows 3--5 orthogonality}
\[
\begin{array}{c|c|c}
\text{Pair} & \text{Missing cross term} & \text{Justification} \\ \hline
\mathbf{v}_{32} \perp \mathbf{v}_{34} & x_3^2 y_2 y_4 & \text{row 3} \\
\mathbf{v}_{42} \perp \mathbf{v}_{43} & x_4^2 y_2 y_3 & \text{row 4} \\
\mathbf{v}_{42} \perp \mathbf{v}_{44} & x_4^2 y_2 y_4 & \text{row 4} \\
\mathbf{v}_{43} \perp \mathbf{v}_{44} & x_4^2 y_3 y_4 & \text{row 4} \\
\mathbf{v}_{51} \perp \mathbf{v}_{52} & x_5^2 y_1 y_2 & \text{row 5} \\
\mathbf{v}_{51} \perp \mathbf{v}_{54} & x_5^2 y_1 y_4 & \text{row 5} \\
\mathbf{v}_{52} \perp \mathbf{v}_{54} & x_5^2 y_2 y_4 & \text{row 5}
\end{array}
\]

%For $i=2$: missing $x_2^2 y_1 y_3$, $x_2^2 y_1 y_4$, $x_2^2 y_3 y_4$ give
%\[
%\mathbf{v}_{21} \perp \mathbf{v}_{23},\quad \mathbf{v}_{21} \perp \mathbf{v}_{24},\quad \mathbf{v}_{23} \perp \mathbf{v}_{24}.
%\]
%
%For $i=3$: missing $x_3^2 y_2 y_4$ gives $\mathbf{v}_{32} \perp \mathbf{v}_{34}$.
%
%For $i=4$: missing $x_4^2 y_2 y_3$, $x_4^2 y_2 y_4$, $x_4^2 y_3 y_4$ give
%\[
%\mathbf{v}_{42} \perp \mathbf{v}_{43},\quad \mathbf{v}_{42} \perp \mathbf{v}_{44},\quad \mathbf{v}_{43} \perp \mathbf{v}_{44}.
%\]

%For $i=5$: missing $x_5^2 y_1 y_2$, $x_5^2 y_1 y_4$, $x_5^2 y_2 y_4$ give
%\[
%\mathbf{v}_{51} \perp \mathbf{v}_{52},\quad \mathbf{v}_{51} \perp \mathbf{v}_{54},\quad \mathbf{v}_{52} \perp \mathbf{v}_{54}.
%\]

\noindent\textit{Between rows.}
We list all missing cross terms $x_i x_k y_j y_l$ ($i < k$) that yield orthogonality relations among the vectors in $$\mathcal{S}=\{\mathbf{v}_{11},\mathbf{v}_{12},\mathbf{v}_{13},\mathbf{v}_{21},\mathbf{v}_{23},\mathbf{v}_{24}, \mathbf{v}_{32},\mathbf{v}_{34},\mathbf{v}_{42},\mathbf{v}_{43},\mathbf{v}_{44},\mathbf{v}_{51}\}.$$ For each such term, the equation $\mathbf{v}_{ij} \cdot \mathbf{v}_{kl} + \mathbf{v}_{il} \cdot \mathbf{v}_{kj} = 0$ (or $\mathbf{v}_{ij} \cdot \mathbf{v}_{kj} = 0$ when $j=l$) simplifies using zero vectors to give $\mathbf{v}_{ij} \perp \mathbf{v}_{kl}$ or $\mathbf{v}_{il} \perp \mathbf{v}_{kj}$.

The following tables enumerate all pairwise orthogonality relations among the distinct vectors. Each entry indicates that the two vectors are orthogonal, with the missing cross term that forces it. \\
\textbf{Cross orthogonality: $\mathbf{v}_{11}$ with others}
\[
\begin{array}{c|c|c}
\text{Pair} & \text{Missing cross term} & \text{Justification} \\ \hline
\mathbf{v}_{11} \perp \mathbf{v}_{21} & x_1 x_2 y_1^2 & \text{same }y_1 \\
\mathbf{v}_{11} \perp \mathbf{v}_{23} & x_1 x_5 y_1 y_4 & \mathbf{v}_{54}=\mathbf{v}_{23},\ \mathbf{v}_{14}=0 \\
\mathbf{v}_{11} \perp \mathbf{v}_{24} & x_1 x_2 y_1 y_4 & \mathbf{v}_{14}=0 \\
\mathbf{v}_{11} \perp \mathbf{v}_{32} & x_1 x_3 y_1 y_2 & \mathbf{v}_{31}=0 \\
\mathbf{v}_{11} \perp \mathbf{v}_{34} & x_1 x_3 y_1 y_4 & \mathbf{v}_{14}=0,\ \mathbf{v}_{31}=0 \\
\mathbf{v}_{11} \perp \mathbf{v}_{42} & x_1 x_4 y_1 y_2 & \mathbf{v}_{41}=0 \\
\mathbf{v}_{11} \perp \mathbf{v}_{43} & x_1 x_4 y_1 y_3 & \mathbf{v}_{41}=0 \\
\mathbf{v}_{11} \perp \mathbf{v}_{44} & x_1 x_4 y_1 y_4 & \mathbf{v}_{14}=0,\ \mathbf{v}_{41}=0 \\
\mathbf{v}_{11} \perp \mathbf{v}_{51} & x_1 x_5 y_1^2 & \text{same }y_1
\end{array}
\]
\textbf{Cross orthogonality: $\mathbf{v}_{12}$ with others}
\[
\begin{array}{c|c|c}
\text{Pair} & \text{Missing cross term} & \text{Justification} \\ \hline
\mathbf{v}_{12} \perp \mathbf{v}_{21} & x_1 x_2 y_1 y_2 & \mathbf{v}_{22}=0 \\
\mathbf{v}_{12} \perp \mathbf{v}_{23} & x_1 x_2 y_2 y_3 & \mathbf{v}_{22}=0 \\
\mathbf{v}_{12} \perp \mathbf{v}_{24} & x_1 x_2 y_2 y_4 & \mathbf{v}_{14}=0,\,\mathbf{v}_{22}=0 \\
\mathbf{v}_{12} \perp \mathbf{v}_{32} & x_1 x_3 y_2^2 & \text{same }y_2 \\
\mathbf{v}_{12} \perp \mathbf{v}_{34} & x_1 x_3 y_2 y_4 & \mathbf{v}_{14}=0 \\
\mathbf{v}_{12} \perp \mathbf{v}_{42} & x_1 x_4 y_2^2 & \text{same }y_2 \\
\mathbf{v}_{12} \perp \mathbf{v}_{43} & x_1 x_4 y_2 y_3 & \mathbf{v}_{53}=0,\,\mathbf{v}_{13} \perp \mathbf{v}_{52},\, \mathbf{v}_{52}=\mathbf{v}_{42} \\
\mathbf{v}_{12} \perp \mathbf{v}_{44} & x_1 x_4 y_2 y_4 & \mathbf{v}_{14}=0 \\
\mathbf{v}_{12} \perp \mathbf{v}_{51} & x_1 x_5 y_1 y_2 & \mathbf{v}_{11} \perp \mathbf{v}_{42},\ \mathbf{v}_{52}=\mathbf{v}_{42}
\end{array}
\]
\textbf{Cross orthogonality: $\mathbf{v}_{13}$ with others}
\[
\begin{array}{c|c|c}
\text{Pair} & \text{Missing cross term} & \text{Justification} \\ \hline
\mathbf{v}_{13} \perp \mathbf{v}_{21} & x_1 x_2 y_1 y_3 & \mathbf{v}_{11} \perp \mathbf{v}_{23} \\
\mathbf{v}_{13} \perp \mathbf{v}_{23} & x_1 x_5 y_3 y_4 & \mathbf{v}_{54}=\mathbf{v}_{23},\ \mathbf{v}_{14}=0,\ \mathbf{v}_{53}=0 \\
\mathbf{v}_{13} \perp \mathbf{v}_{24} & x_1 x_2 y_3 y_4 & \mathbf{v}_{14}=0 \\
\mathbf{v}_{13} \perp \mathbf{v}_{32} & x_1 x_3 y_2 y_3 & \mathbf{v}_{33}=0 \\
\mathbf{v}_{13} \perp \mathbf{v}_{34} & x_1 x_3 y_3 y_4 & \mathbf{v}_{14}=0,\ \mathbf{v}_{33}=0 \\
\mathbf{v}_{13} \perp \mathbf{v}_{42} & x_1 x_5 y_2 y_3 & \mathbf{v}_{53}=0,\ \mathbf{v}_{52}=\mathbf{v}_{42} \\
\mathbf{v}_{13} \perp \mathbf{v}_{43} & x_1 x_4 y_3^2   & \text{same }y_3 \\
\mathbf{v}_{13} \perp \mathbf{v}_{44} & x_1 x_4 y_3 y_4 & \mathbf{v}_{14}=0 \\
\mathbf{v}_{13} \perp \mathbf{v}_{51} & x_1 x_5 y_1 y_3 & \mathbf{v}_{53}=0
\end{array}
\]
\textbf{Cross orthogonality: $\mathbf{v}_{21}$ with others}
\[
\begin{array}{c|c|c}
\text{Pair} & \text{Missing cross term} & \text{Justification} \\ \hline
\mathbf{v}_{21} \perp \mathbf{v}_{32} & x_2 x_3 y_1 y_2 & \mathbf{v}_{22}=0,\ \mathbf{v}_{31}=0 \\
\mathbf{v}_{21} \perp \mathbf{v}_{34} & x_2 x_3 y_1 y_4 & \mathbf{v}_{31}=0 \\
\mathbf{v}_{21} \perp \mathbf{v}_{42} & x_2 x_4 y_1 y_2 & \mathbf{v}_{22}=0,\ \mathbf{v}_{41}=0 \\
\mathbf{v}_{21} \perp \mathbf{v}_{43} & x_2 x_4 y_1 y_3 & \mathbf{v}_{41}=0 \\
\mathbf{v}_{21} \perp \mathbf{v}_{44} & x_2 x_4 y_1 y_4 & \mathbf{v}_{41}=0 \\
\mathbf{v}_{21} \perp \mathbf{v}_{51} & x_2 x_5 y_1^2   & \text{same }y_1
\end{array}
\]
\textbf{Cross orthogonality: $\mathbf{v}_{23}$ with others}
\[
\begin{array}{c|c|c}
\text{Pair} & \text{Missing cross term} & \text{Justification} \\ \hline
\mathbf{v}_{23} \perp \mathbf{v}_{32} & x_2 x_3 y_2 y_3 & \mathbf{v}_{22}=0,\,\mathbf{v}_{33}=0 \\
\mathbf{v}_{23} \perp \mathbf{v}_{34} & x_2 x_3 y_3 y_4 & \mathbf{v}_{33}=0 \\
\mathbf{v}_{23} \perp \mathbf{v}_{42} & x_2 x_4 y_2 y_3 & \mathbf{v}_{22}=0 \\
\mathbf{v}_{23} \perp \mathbf{v}_{43} & x_2 x_4 y_3^2   & \text{same }y_3 \\
\mathbf{v}_{23} \perp \mathbf{v}_{44} & x_5 x_4 y_4^2   & \mathbf{v}_{23}=\mathbf{v}_{54},\, \text{same }y_4 \\
\mathbf{v}_{23} \perp \mathbf{v}_{51} & x_2 x_5 y_1 y_3 & \mathbf{v}_{53}=0
\end{array}
\]
\textbf{Cross orthogonality: $\mathbf{v}_{24}$ with others}
\[
\begin{array}{c|c|c}
\text{Pair} & \text{Missing cross term} & \text{Justification} \\ \hline
\mathbf{v}_{24} \perp \mathbf{v}_{32} & x_2 x_3 y_2 y_4 & \mathbf{v}_{22}=0 \\
\mathbf{v}_{24} \perp \mathbf{v}_{34} & x_2 x_3 y_4^2   & \text{same }y_4 \\
\mathbf{v}_{24} \perp \mathbf{v}_{42} & x_2 x_4 y_2 y_4 & \mathbf{v}_{22}=0 \\
\mathbf{v}_{24} \perp \mathbf{v}_{43} & x_2 x_4 y_3 y_4 & \mathbf{v}_{23} \perp \mathbf{v}_{44} \\
\mathbf{v}_{24} \perp \mathbf{v}_{44} & x_2 x_4 y_4^2   & \text{same }y_4\\
\mathbf{v}_{24} \perp \mathbf{v}_{51} & x_2 x_5 y_1 y_4 & \mathbf{v}_{54}=\mathbf{v}_{23},\ \mathbf{v}_{21} \perp \mathbf{v}_{23}
\end{array}
\]
\textbf{Cross orthogonality: $\mathbf{v}_{32}$ with others}
\[
\begin{array}{c|c|c}
\text{Pair} & \text{Missing cross term} & \text{Justification} \\ \hline
\mathbf{v}_{32} \perp \mathbf{v}_{42} & x_3 x_4 y_2^2 & \text{same }y_2 \\
\mathbf{v}_{32} \perp \mathbf{v}_{43} & x_3 x_4 y_2 y_3 & \mathbf{v}_{33}=0 \\
\mathbf{v}_{32} \perp \mathbf{v}_{44} & x_3 x_4 y_2 y_4 & \mathbf{v}_{32}\perp\mathbf{v}_{23},\, \mathbf{v}_{23}=\mathbf{v}_{54},\,\mathbf{v}_{34}\perp\mathbf{v}_{52},\,\mathbf{v}_{52}=\mathbf{v}_{42}\\
\mathbf{v}_{32} \perp \mathbf{v}_{51} & x_3 x_5 y_1 y_2 & \mathbf{v}_{31}=0
\end{array}
\]
\textbf{Cross orthogonality: $\mathbf{v}_{34}$ with others}
\[
\begin{array}{c|c|c}
\text{Pair} & \text{Missing cross term} & \text{Justification} \\ \hline
\mathbf{v}_{34} \perp \mathbf{v}_{42} & x_3 x_4 y_2 y_4 & \mathbf{v}_{32} \perp \mathbf{v}_{44} \\
\mathbf{v}_{34} \perp \mathbf{v}_{43} & x_3 x_4 y_3 y_4 & \mathbf{v}_{33}=0 \\
\mathbf{v}_{34} \perp \mathbf{v}_{44} & x_3 x_4 y_4^2   & \text{same }y_4 \\
\mathbf{v}_{34} \perp \mathbf{v}_{51} & x_3 x_5 y_1 y_4 & \mathbf{v}_{31}=0
\end{array}
\]
\textbf{Cross orthogonality: $\mathbf{v}_{42}$, $\mathbf{v}_{43}$, $\mathbf{v}_{44}$ with $\mathbf{v}_{51}$}
\[
\begin{array}{c|c|c}
\text{Pair} & \text{Missing cross term} & \text{Justification} \\ \hline
\mathbf{v}_{42} \perp \mathbf{v}_{51} & x_4 x_5 y_1 y_2 & \mathbf{v}_{41}=0 \\
\mathbf{v}_{43} \perp \mathbf{v}_{51} & x_4 x_5 y_1 y_3 & \mathbf{v}_{41}=0 \\
\mathbf{v}_{44} \perp \mathbf{v}_{51} & x_4 x_5 y_1 y_4 & \mathbf{v}_{41}=0
\end{array}
\]

A systematic verification confirms that every pair of distinct vectors in the set $\mathcal{S}$
%\[
%\mathcal{S} = \{\mathbf{v}_{11}, \mathbf{v}_{12}, \mathbf{v}_{13}, \mathbf{v}_{21}, \mathbf{v}_{23}, \mathbf{v}_{24}, %\mathbf{v}_{32}, \mathbf{v}_{34}, \mathbf{v}_{42}, \mathbf{v}_{43}, \mathbf{v}_{44}, \mathbf{v}_{51}\}
%\]
is orthogonal. Thus $\mathcal{S}$ consists of $12$ nonzero pairwise orthogonal vectors in $\mathbb{R}^r$, which forces $r \ge 12$.

Since we have an explicit decomposition into $12$ squares, the minimal rank satisfies $12 \le \operatorname{sos}(P) \le 12$, hence $\operatorname{sos}(P) = 12$.
\end{proof}

\begin{remark}
The form $P$ in Definition~\ref{def:5x4-form} is neither simple (it contains cross terms) nor $y$-deficient (the variable $y_2$ appears in the cross term $2x_4 x_5 y_2^2$). The proof demonstrates that the orthogonality method applies robustly even when cross terms with identical $y$-indices are present, as long as the sparsity pattern yields enough missing cross terms to force orthogonality among the distinct vectors.
\end{remark}

\begin{remark}
The explicit SOS decomposition given in the proof uses $12$ squares: ten pure squares $(x_i y_j)^2$ together with $(x_2 y_3 + x_5 y_4)^2$ and $(x_4 y_2 + x_5 y_2)^2$. This matches the lower bound, confirming that the decomposition is optimal.
\end{remark}

\bigskip

\noindent\textbf{Acknowledgement}
% This work was partially supported by Research Center for Intelligent Operations Research, The Hong Kong Polytechnic University (4-ZZT8), the National Natural Science Foundation of China (Nos. 12471282, 12471351 and 12131004), and Jiangsu Provincial Scientific Research Center of Applied Mathematics (Grant No. BK20233002).

This work was partially supported by the National Natural Science Foundation of China (Nos. 12471351).

\medskip

\noindent\textbf{Data availability}
No datasets were generated or analysed during the current study.

\medskip

\noindent\textbf{Conflict of interest} The authors declare no conflict of interest.

\end{document}